\documentclass[10pt,leqno]{amsart}
\usepackage{indentfirst, amssymb,amsmath,amsthm,color}
\usepackage{csquotes}
\usepackage{hyperref}
\usepackage{graphicx}
\newtheorem{theo}{Theorem}[section]

\newtheorem{note}{Note}[section]
\numberwithin {equation}{section}
\begin{document}
\title[An integral representation of the generalized Fibonacci numbers]{An integral representation of the generalized Fibonacci numbers}
\date{}
\author[B. Chakraborty, K. Barman, S. Roy and R. Sinha]{Bikash Chakraborty, Kalyan Barman, Soumon Roy and Ritam Sinha}
\date{}
\address{{\bf Bikash Chakraborty}, Nevanlinna Lab, Department of Mathematics, Ramakrishna Mission Vivekananda Centenary College, Rahara, West Bengal 700118, India.}
\email{bikashchakraborty.math@yahoo.com, bikashchakrabortyy@gmail.com}
\address{{\bf Kalyan Barman}, Nevanlinna Lab, Department of Mathematics, Ramakrishna Mission Vivekananda Centenary College, Rahara, West Bengal 700118, India.}
\email{kalyannbu123@gmail.com}
\address{{\bf Soumon Roy}, Nevanlinna Lab, Department of Mathematics, Ramakrishna Mission Vivekananda Centenary College, Rahara, West Bengal 700118, India.}
\email{rsoumon@gmail.com}
\address{{\bf Ritam Sinha}, Nevanlinna Lab, Department of Mathematics, Ramakrishna Mission Vivekananda Centenary College, Rahara, West Bengal 700118, India.}
\email{ritamsinha23@gmail.com}
\footnotetext{\textbf{Keywords}: Fibonacci numbers; Tribonacci  numbers; Contour integral; Cauchy's theorem, $k$-generalized Fibonacci numbers, Binet form.}
\footnotetext{\textbf{Mathematics Subject Classification}: Primary 11B39, 97I80; Secondary 05A19, 11B83, 11B99, 30E20.}
\maketitle
\begin{abstract}
In this note, we present an integral representation of the $k$-generalized Fibonacci sequence and provide a proof using contour integration. This representation enables the derivation of several known identities and a Binet-type formula related to the $k$-generalized Fibonacci sequence.
\end{abstract}
\section{Fibonacci Numbers}
The Fibonacci sequence is defined such that each term is the sum of the two preceding terms. These numbers, known as Fibonacci numbers and denoted by $F_n$, typically begin with $0$ and $1$. Starting from these initial values, the sequence progresses as follows:
$$0, 1, 1, 2, 3, 5, 8, 13, 21, 34, 55, 89, 144, \ldots$$
The Fibonacci numbers possess a wide range of well-studied properties, as evidenced by the two comprehensive volumes devoted to them by Koshy \cite{K1, K2}.\\
Recently, Andrica and Bagdasar (\cite{AB}, p. 132)  as well as Glasser and Zhou \cite{GZ} have provided two integral representations for the Fibonacci numbers. If $F_n$ is the $n$-th Fibonacci number, then
$$F_n=\frac{1}{2\pi R^{n-1}}\int_{0}^{2\pi}\frac{-R^2\cos(n+1)t-R\cos nt+\cos(n-1)t}{R^4+3R^2+1+2(R^3-R)\cos t-4R^2\cos^2 t}dt,$$
for every positive real number $R<\frac{\sqrt{5}-1}{2}$ (see \cite{AB}, p. 132). Also,
$$F_n=\frac{1}{\sqrt{5}}\left(\frac{\sqrt{5}+1}{2}\right)^n-\frac{2}{\pi}\int_{0}^{\infty}\frac{\sin \frac{x}{2}}{x}\frac{\cos(nx)-2\sin(nx)\sin x}{5\sin^2 x+\cos^2 x} dx,$$
for $n\in \mathbb{N}\cup\{0\}$, \cite{GZ}. Additionally, Dilcher \cite{D} offered an integral representation for the even Fibonacci numbers as:
$$F_{2n}=\frac{n}{2^n}\int_{-1}^{1}(3+x\sqrt{5})^{n-1}dx,$$
where $n\in \mathbb{N}$. In a recent note, Stewart \cite{S} obtained that
$$F_{n}=\frac{n}{2^n}\int_{-1}^{1}(3+x\sqrt{5})^{n-1}dx,$$
where $n\in \mathbb{N}$. Additionally, replacing $n$ by $2n$ in the above formula, in \cite{S}, Stewart obtained the same formula for even Fibonacci numbers as Dilcher  offered in \cite{D}.\medbreak
In this note, we present a new integral representation of the Fibonacci numbers using Fibonacci polynomial $z^2-z-1$ and contour integration. To be more precise, using the classical Fibonacci numbers as an example, we begin with their generating function
$$\sum_{k=0}^{\infty} F_kz^k =\frac{z}{1-z-z^2}.$$
Now, differentiating both sides $n$ times and evaluating at $z=0$, and then applying Cauchy's integral formula for derivatives, we obtain
\begin{eqnarray*}
  n!\cdot F_n &=& \bigg[\frac{d^n}{dz^n} \left(\sum_{k=0}^{\infty} F_kz^k\right)\bigg]_{z=0} \\
  &=& \bigg[\frac{d^n}{dz^n} \left(\frac{z}{1-z-z^2}\right)\bigg]_{z=0} \\
  &=& \frac{n!}{2\pi i}\int_{\gamma}\frac{z/(1-z-z^2)}{(z-0)^{n+1}}\;dz\\
  &=& \frac{n!}{2\pi i}\int_{\gamma}\frac{z^{-n}}{1-z-z^2}\;dz,
\end{eqnarray*}
where $\gamma$  is the curve traversing a circle of sufficiently small radius (so that $z/(1-z-z^2)$ become holomorphic inside the $\gamma$) once around
the origin in the positive direction.\\ In the above integral, the substitution $z\to z^{-1}$ motivates us to propose a new integral representation of the Fibonacci numbers as follows:
\begin{eqnarray}
  F_n &:=& \frac{1}{2\pi i}\int\limits_{|z|=\phi^+}\frac{z^n}{z^2-z-1}\; dz,
\end{eqnarray}
where $\phi^+$ is any real number strictly greater than $\frac{1+\sqrt{5}}{2}$. Let $\alpha$ and $\beta$ be the two zeros of $z^2-z-1=0$. Using the Residue theorem \cite{Co}, we obtain
\begin{enumerate}
  \item [(i)] $F_0=\frac{1}{2\pi i}\int\limits_{|z|=\phi^+}\frac{dz}{z^2-z-1}=\frac{1}{\alpha-\beta}+\frac{1}{\beta-\alpha}=0$,
  \item [(ii)] $F_1=\frac{1}{2\pi i}\int\limits_{|z|=\phi^+}\frac{z}{z^2-z-1}\;dz=\frac{\alpha}{\alpha-\beta}+\frac{\beta}{\beta-\alpha}=1$, and
  \item [(iii)] for $n\geq 2$, using the Cauchy's theorem \cite{Co}, we have $$F_n-F_{n-1}-F_{n-2}=\frac{1}{2\pi i}\int\limits_{|z|=\phi^+}\frac{(z^n-z^{n-1}-z^{n-2})}{z^2-z-1}\;dz=\frac{1}{2\pi i}\int\limits_{|z|=\phi^+}z^{n-2}\;dz=0,$$
      i.e., $F_n=F_{n-1}+F_{n-2}$ for $n\geq 2$.
\end{enumerate}
Thus the proposed integral represents the $n$th Fibonacci number.
\section{$k$- generalized Fibonacci Numbers}
For  a fixed integer $k\geq 2$, the $n$th $k$-generalized Fibonacci sequences \cite {DD, lee}, denoted by $F_n^{(k)}$, is defined by the following recurrence equations:
$$F_n^{(k)}=\Bigg\{\begin{array}{c}
              0~~\text{if}~~0\leq n \leq k-2,\hfill \\
              1~~\text{if}~~n=k-1,\hfill \\
              F_{n-1}^{(k)}+F_{n-2}^{(k)}+\cdots+F_{n-k}^{(k)}~~\text{if}~~n\geq k.
            \end{array}$$
Next, we present a chart from \cite{W} of the initial values of $k$-generalized Fibonacci numbers  for small $k$'s:

\begin{center}
   \begin{tabular}{|c|c|c|c|}
     \hline
     $\textbf{$k$}$ & \textbf{Name} & \textbf{First few terms} \\
     \hline
     2 & Fibonacci & 0, 1, 1, 2, 3, 5, 8, 13, 21, 34, 55, 89, 144, 233, 377, 610, 987, 1597, 2584, \ldots\\
     \hline
     3 & Tribonacci & 0, 0, 1, 1, 2, 4, 7, 13, 24, 44, 81, 149, 274, 504, 927, 1705, 3136, 5768, \ldots \\
     \hline
     4 & Tetranacci & 0, 0, 0, 1, 1, 2, 4, 8, 15, 29, 56, 108, 208, 401, 773, 1490, 2872, 5536, \ldots \\
     \hline

     5 & Pentanacci & 0, 0, 0, 0, 1, 1, 2, 4, 8, 16, 31, 61, 120, 236, 464, 912, 1793, 3525, \ldots \\
     \hline
     6 & Hexanacci & 0, 0, 0, 0, 0, 1, 1, 2, 4, 8, 16, 32, 63, 125, 248, 492, 976, 1936, \ldots \\
     \hline
   \end{tabular}
   
   {\center{\textbf{Table-1}: Initial values of $k$-generalized Fibonacci numbers  for small $k$'s}}
  \end{center}
  \medbreak
To study the $k$-generalized Fibonacci numbers, the polynomial $z^k-z^{k-1}-z^{k-2}-\cdots-z-1$ plays a crucial role. Now, we recall some important properties of this polynomial.

\begin{theo}\cite{M1, M2}\label{120625}
  Let $f(z) =z^k-z^{k-1}-\cdots-z-1$, for $k\geq 2$. Then
  \begin{enumerate}
    \item [(a)] $f$ has a real zero $z_0$ such that $1 < z_0 < 2$ ;
    \item [(b)] the remaining $k- 1$ zeros of $f$ lie within the unit circle in the complex plane;
    \item [(c)] the zeros of $f$ are simple.
  \end{enumerate}
\end{theo}
Let the zeros of the polynomial equation $z^k-z^{k-1}-z^{k-2}-\cdots-z-1=0$ be denoted by $\alpha_1, \alpha_2,\cdots,\alpha_k$. By Theorem \ref{120625}, one of the $\alpha_i$'s is real and strictly lies between  one and two.  Without loss of generality, we assume that $1< \alpha_1 <2$. Hence, by Theorem \ref{120625}, we have $\mid \alpha_j\mid \leq 1$ for $2\leq j \leq k$. With this understanding, we define $\phi:= \alpha_1$.\\
Analogous to the classical Fibonacci sequence, a formal power series can be constructed in which the coefficients correspond to the $k$-generalized Fibonacci numbers:
$$F(z)=\sum_{n=0}^{\infty} F_n^{(k)}z^n=\sum_{n=k-1}^{\infty} F_n^{(k)}z^n,$$
as $F_0^{(k)}=0,~F_1^{(k)}=0,~\cdots,F_{k-2}^{(k)}=0$. Now, by separating initial $k$ terms from the sum and substitute the recurrence relation for $F_n^{(k)}$ into the coefficients of the sum, we can get
$$F(z)=z^{k-1}+zF(z)+z^2F(z)+\cdots+z^kF(z),$$
which gives the corresponding generating function of the  $k$-generalized Fibonacci numbers as
$$\sum_{n=0}^{\infty} F_n^{(k)}z^n =\frac{z^{k-1}}{1-z-z^2-\cdots-z^k}.$$
In a manner analogous to the classical Fibonacci numbers, the generating function naturally suggests an integral representation for the $n$th term of the $k$-generalized Fibonacci sequence, as follows:
\begin{eqnarray}\label{7o1}
  F_n^{(k)}:=\frac{1}{2\pi i}\int\limits_{|z|=\phi^+}\frac{z^n}{z^k-z^{k-1}-z^{k-2}-\cdots-z-1}\; dz,
\end{eqnarray}
where $n\geq 0$ and $k\geq 2$ are two integers and $\phi^+$ is any real number strictly greater than $\alpha_1$. On simplification, (\ref{7o1}) can be written as
\begin{eqnarray}\label{7o2}
  F_n^{(k)}:=\frac{1}{2\pi i}\int\limits_{|z|=\phi^+}\frac{z^n(z-1)}{z^{k+1}-2z^{k}+1}\; dz.
\end{eqnarray}
\begin{theo}
  Let $  F_n^{(k)}$ be the $n$th $k$-generalized Fibonacci numbers defined by (\ref{7o1}). Then $  F_n^{(k)}=0$ for $0\leq n\leq k-2$.
\end{theo}
\begin{proof}
  Let $t(>\phi^+)$ be a real number. For $\mid z\mid =t$, we have
  \begin{eqnarray*}
 &&  \left| \frac{z^n}{z^k-z^{k-1}-z^{k-2}-\cdots-z-1} \right|\\
 &\leq& \frac{\mid z\mid ^n}{\mid z\mid^k-\mid z\mid ^{k-1}-\mid z\mid ^{k-2}-\cdots-\mid z\mid-1} \\
 &=& \frac{\mid t\mid ^n}{\mid t\mid^k-\mid t\mid ^{k-1}-\mid t\mid ^{k-2}-\cdots-\mid t\mid-1}.
  \end{eqnarray*}
  Thus by the M-L inequality \cite{Co}, we have
    \begin{eqnarray*}
 \left|F_n^{(k)}\right| &=& \left| \frac{1}{2\pi i}\int\limits_{|z|=t}\frac{z^n}{z^k-z^{k-1}-z^{k-2}-\cdots-z-1}\; dz\right|\\
 &\leq& \frac{1}{2\pi}\cdot\frac{\mid t\mid ^n}{\mid t\mid^k-\mid t\mid ^{k-1}-\mid t\mid ^{k-2}-\cdots-\mid t\mid-1}\cdot2\pi t \\
 &=& \frac{\mid t\mid ^{n+1}}{\mid t\mid^k-\mid t\mid ^{k-1}-\mid t\mid ^{k-2}-\cdots-\mid t\mid-1}\to 0~~\text{as}~~t\to\infty~~\text{when}~~ 0\leq n\leq k-2.
  \end{eqnarray*}
  Thus $ F_n^{(k)}=0$ for $0\leq n\leq k-2$ and $k\geq 2$. This completes the proof.
\end{proof}
\begin{theo}
  Let $  F_n^{(k)}$ be the $n$th $k$-generalized Fibonacci numbers defined by (\ref{7o1}). Then $  F_{k-1}^{(k)}=1$.
\end{theo}
\begin{proof}
  Let $t(>\phi^+)$ be a real number. For $\mid z\mid =t$, we have
  \begin{eqnarray*}
 &&  \left| \frac{z^{k-1}}{z^k-z^{k-1}-z^{k-2}-\cdots-z-1}-\frac{1}{z} \right|\\
 &=&  \left| \frac{z^{k-1}+z^{k-2}+\cdots+z+1}{z(z^k-z^{k-1}-z^{k-2}-\cdots-z-1)} \right|\\
 &\leq& \frac{\mid z\mid ^{k-1}+\mid z\mid ^{k-2}+\cdots+\mid z\mid+1}{\mid z\mid (\mid z\mid^k-\mid z\mid ^{k-1}-\mid z\mid ^{k-2}-\cdots-\mid z\mid-1)} \\
&=& \frac{\mid t\mid ^{k-1}+\mid t\mid ^{k-2}+\cdots+\mid t\mid+1}{\mid t\mid (\mid t\mid^k-\mid t\mid ^{k-1}-\mid t\mid ^{k-2}-\cdots-\mid t\mid-1)}.
  \end{eqnarray*}
  Thus by the M-L inequality \cite{Co}, we have
    \begin{eqnarray*}
 \left|F_{k-1}^{(k)}-1\right| &=&  \left| \frac{1}{2\pi i}\int\limits_{|z|=t}\frac{z^{k-1}}{z^k-z^{k-1}-z^{k-2}-\cdots-z-1}\; dz-\frac{1}{2\pi i}\int\limits_{|z|=t}\frac{1}{z}\;dz\right|\\
 &\leq& \frac{1}{2\pi}\cdot\frac{\mid t\mid ^{k-1}+\mid t\mid ^{k-2}+\cdots+\mid t\mid+1}{\mid t\mid (\mid t\mid^k-\mid t\mid ^{k-1}-\mid t\mid ^{k-2}-\cdots-\mid t\mid-1)}\cdot2\pi t \\
 &=& \frac{\mid t\mid ^{k-1}+\mid t\mid ^{k-2}+\cdots+\mid t\mid+1}{(\mid t\mid^k-\mid t\mid ^{k-1}-\mid t\mid ^{k-2}-\cdots-\mid t\mid-1)}\to 0~~\text{as}~~t\to\infty.
  \end{eqnarray*}
  Thus $  F_{k-1}^{(k)}=1$. This completes the proof.
\end{proof}
\begin{theo}
Let $  F_n^{(k)}$ be the $n$th $k$-generalized Fibonacci numbers defined by (\ref{7o1}). Then
  $$F_n^{(k)}=F_{n-1}^{(k)}+F_{n-2}^{(k)}+\cdots+F_{n-k}^{(k)}~~\text{if}~~n\geq k.$$
\end{theo}
\begin{proof}
 For $z\not\in\{1, \alpha_1, \alpha_2,\cdots, \alpha_k\}$, we have
  \begin{eqnarray*}
   \frac{z^n-(z^{n-1}+z^{n-2}+\cdots+z^{n-k})}{z^k-z^{k-1}-z^{k-2}-\cdots-z-1}
    &=& z^{n-k}.
  \end{eqnarray*}
  Thus
   \begin{eqnarray*}
    && \frac{1}{2\pi i}\int\limits_{\mid z\mid =\phi^+}\frac{z^n}{z^k-z^{k-1}-z^{k-2}-\cdots-z-1}\;dz\\
    &=&\sum_{j=1}^{k}\frac{1}{2\pi i}\int\limits_{\mid z\mid =\phi^+}\frac{z^{n-j}}{z^k-z^{k-1}-z^{k-2}-\cdots-z-1}\;dz+ \frac{1}{2\pi i}\int_{\mid z\mid =\phi+}z^{n-k}\;dz.
  \end{eqnarray*}
  Thus by the Cauchy's theorem \cite{Co}, we have
   $$F_n^{(k)}=F_{n-1}^{(k)}+F_{n-2}^{(k)}+\cdots+F_{n-k}^{(k)}.$$
   This completes the proof.
\end{proof}
Next, we observe a Binet-type formula \cite{DD, Sp} that can be used to generate the $n$th term of the $k$-generalized Fibonacci numbers. Although this formula is already known in the literature, we present a new proof using this integral representation.
\begin{theo}
Let $  F_n^{(k)}$ be the $n$th $k$-generalized Fibonacci numbers. Then $$ F_n^{(k)}=\sum\limits_{j=1}^{k}\frac{\alpha_j^n}{\prod\limits_{\substack{1\leq i\leq k\\ i\not=j}}(\alpha_j-\alpha_i)}
,$$ where  $n\geq 0$ and $ k\geq 2$ are two integers and  $\alpha_1, \alpha_2,\cdots,\alpha_k$ are the zeros of  $z^k-z^{k-1}-z^{k-2}-\cdots-z-1=0$.
\end{theo}
\begin{proof} Using the Residue theorem \cite{Co}, we have
  \begin{eqnarray*}
   F_n^{(k)} &=& \frac{1}{2\pi i}\int\limits_{|z|=\phi^+}\frac{z^n}{z^k-z^{k-1}-z^{k-2}-\cdots-z-1}\; dz \\
     &=& \frac{1}{2\pi i}\int\limits_{|z|=\phi^+}\frac{z^n}{\prod_{j=1}^{k}(z-\alpha_j)}\; dz \\
     &=& \sum\limits_{j=1}^{k}\frac{\alpha_j^n}{\prod\limits_{\substack{1\leq i\leq k\\ i\not=j}}(\alpha_j-\alpha_i)}.
  \end{eqnarray*}
This completes the proof.
\end{proof}
Thus for $k=2$, $F_n^{(2)}$ denotes the $n$-th Fibonacci number. If $\alpha, \beta$ are the two zeros of $z^2-z-1=0$, then
\begin{eqnarray*}
  F_n^{(2)} &=& \frac{\beta^n}{\beta-\alpha}+\frac{\alpha^n}{\alpha-\beta}\\
   &=& \frac{\beta^n-\alpha^n}{\beta-\alpha},
\end{eqnarray*}
 which is the famous Binet's formula for the Fibonacci numbers. Next, we prove the Lemma 2 in \cite{BT} using our method.
 \begin{theo}
Let $  F_n^{(k)}$ be the $n$th $k$-generalized Fibonacci numbers. Then for $n\geq k+1$, $$F_n^{(k)}=2F_{n-1}^{(k)}-F_{n-k-1}^{(k)}.$$
 \end{theo}
 \begin{proof}
 For $z\not\in\{1, \alpha_1, \alpha_2,\cdots, \alpha_k\}$, we have
  \begin{eqnarray*}
   \frac{z^n-2z^{n-1}+z^{n-k-1}}{z^k-z^{k-1}-z^{k-2}-\cdots-z-1}
    &=& \frac{z^{n-k-1}(z^{k+1}-2z^{k}+1)}{z^k-\frac{z^k-1}{z-1}} \\
    &=& z^{n-k-1}(z-1).
  \end{eqnarray*}
  Thus by the Cauchy's theorem \cite{Co}, we have
    \begin{eqnarray*}
    \frac{1}{2\pi i}\int_{\mid z\mid =\phi^+} \frac{z^n-2z^{n-1}+z^{n-k-1}}{z^k-z^{k-1}-z^{k-2}-\cdots-z-1}\;dz
       =\frac{1}{2\pi i}\int_{\mid z\mid =\phi^+} z^{n-k-1}(z-1)\;dz=0.
  \end{eqnarray*}
i.e.,
   $$F_n^{(k)}=2F_{n-1}^{(k)}-F_{n-k-1}^{(k)}.$$
   This completes the proof.
\end{proof}

\section{Some identities involving $2$-generalized Fibonacci numbers}
In this section, we revisit some well-known identities given in \cite{W} involving the 2-generalized Fibonacci numbers, i.e., the classical Fibonacci numbers, using their integral representation as described in Section \ref{7o1}. The $n$th Fibonacci number can be expressed as
$$F_n=F_n^{(2)}=\frac{1}{2\pi i}\int_{|z|=\phi^+}\frac{z^n}{z^2-z-1}\; dz,$$
where $\phi^{+}$ is any real number strictly greater than $\frac{1+\sqrt{5}}{2}$. Let $\alpha$ and $\beta$ be the two zeros of $z^2-z-1=0$.
\begin{theo}
If $F_n$ is the $n$th Fibonacci number, then $\sum\limits_{j=1}^{n}F_j=F_{n+2}-1$.
\end{theo}
\begin{proof} Using the integral representation of $n$th Fibonacci number and the residue theorem, we have
 \begin{eqnarray*}
  \nonumber 
  \sum_{j=1}^{n}F_j-F_{n+2} &=& \frac{1}{2\pi i}\int\limits_{|z|=\phi^+} \frac{\sum_{j=1}^{n}z^j-z^{n+2}}{z^2-z-1}\;dz \\
    &=& \frac{1}{2\pi i}\int\limits_{|z|=\phi^+} \frac{z(z^{n}-1)-z^{n+2}(z-1)}{(z-1)(z^2-z-1)}\;dz\\
    &=& \frac{1}{2\pi i}\int\limits_{|z|=\phi^+} \frac{-z^{n+1}(z^{2}-z-1)-z}{(z-1)(z^2-z-1)} \;dz\\
     &=& \frac{1}{2\pi i}\int\limits_{|z|=\phi^+}\left(\frac{-z^{n+1}}{z-1}-\frac{z}{(z-1)(z^2-z-1)}\right)\;dz\\
        &=& -1-\left(-1+\frac{\alpha}{(\alpha-1)(\alpha-\beta)}+\frac{\beta}{(\beta-1)(\beta-\alpha)}\right)\\
         &=& -1,~~\text{as~$\alpha^2-1=\alpha$~and~ $\beta^2-1=\beta$}.
  \end{eqnarray*}
This completes the proof.
\end{proof}
\begin{theo}
If $F_n$ is the $n$th Fibonacci number, then $\sum\limits_{j=0}^{n-1}F_{2j+1}=F_{2n}$.
\end{theo}
\begin{proof} Using the integral representation of $n$th Fibonacci number and the residue theorem, we have
  \begin{eqnarray*}
  \sum_{j=0}^{n-1}F_{2j+1}-F_{2n} &=& \frac{1}{2\pi i}\int\limits_{|z|=\phi+} \frac{\sum_{j=0}^{n-1}z^{2j+1}-z^{2n}}{z^2-z-1}\;dz \\
     &=& \frac{1}{2\pi i}\int\limits_{|z|=\phi^+} \frac{z(z^{2n}-1)-z^{2n}(z^2-1)}{(z^2-1)(z^2-z-1)}\;dz\\
     &=& \frac{1}{2\pi i}\int\limits_{|z|=\phi^+} \frac{-z^{2n}(z^{2}-z-1)-z}{(z^2-1)(z^2-z-1)} \;dz\\
     &=& \frac{1}{2\pi i}\int\limits_{|z|=\phi^+}\left(\frac{-z^{2n}}{z^2-1}-\frac{z}{(z^2-1)(z^2-z-1)}\right)\;dz\\
    &=& -\left(\frac{1}{2}+\frac{1}{-2}\right)-\left(\frac{1}{-2}+\frac{-1}{-2}+\frac{\alpha}{(\alpha^2-1)(\alpha-\beta)}+\frac{\beta}{(\beta^2-1)(\beta-\alpha)}\right)\\
    &=& 0,~~\text{as~$\alpha^2-1=\alpha$,~\text{and}~~$\beta^2-1=\beta$}.
  \end{eqnarray*}
  This completes the proof.
  \end{proof}
  A similar argument gives the following result:
  \begin{theo}
If $F_n$ is the $n$-th Fibonacci number, then $\sum\limits_{i=1}^{n}F_{2i}=F_{2n+1}-1$.
\end{theo}
\begin{note}
 For a nice proof of $F_{n+2}=\sum_{k=0}^{\infty}\binom{n+1-k}{k}$ using the contour integration, see \cite{web}.
\end{note}
\section{Acknowledgement}
We sincerely thank the anonymous referees and the Editor for their careful reading of the manuscript and their valuable suggestions. We are also grateful to Prof. Raymond Mortini for his insightful discussions during the preparation of the initial draft.

\end{document}